\documentclass[reqno]{amsart}
\usepackage{amsmath, amssymb}
 \usepackage{changebar}
\usepackage{enumerate, xspace}
\usepackage[dvips]{graphicx}
\numberwithin{equation}{section}
\usepackage{hyperref}
\newtheorem{thm}{Theorem}[section]
\newtheorem{lem}[thm]{Lemma}
\newtheorem{cor}[thm]{Corollary}

\newtheorem{prop}[thm]{Proposition}

\newtheorem{defin}[thm]{Definition}

\newtheorem{rem}[thm]{Remark}

\renewcommand\Pr{{\mathbb P}}
\newcommand\N{{\mathbb N}}

\newcommand\E{{\mathbb E}}
\newcommand\R{{\mathbb R}}
\newcommand\Z{{\mathbb Z}}
\newcommand\1{{1\kern-.25em\hbox{\rm I}}}
\newcommand\eu{{1\kern-.25em\hbox{\sm I}}}

\newcommand\BB{{\mathcal B}}

\newcommand\FF{{\mathcal F}}

 \newcommand\e{\epsilon}
\newcommand\eps{\epsilon}

\newcommand\om{\omega}
 \newcommand\La{\Lambda}
\newcommand\G{\Gamma}
\newcommand\Om{\Omega}

\newcommand{\grad}{\nabla}
\newcommand{\dist}{{\rm dist}}
\newcommand{\argmin}{{\rm argmin}}
\newcommand{\nada}[1]{}
\begin{document}
 \title[Unique minimizer for a Random functional]
{Unique minimizer for a Random functional with double-well potential   in dimension  1 and 2 }
\author{Nicolas Dirr}\thanks{N.D. supported by GNFM-INDAM }
\address{Nicolas Dirr, Dept. of Mathematical Sciences,
University of Bath,  Bath, BA2 7AY, United Kingdom.}
\email{ {\tt nd235@bath.ac.uk}}
\author{Enza Orlandi}\thanks{E.O. supported by MURST/Cofin  
Prin07: 20078XYHYS 
and ROMA TRE University}
\address{Enza Orlandi, 
Dipartimento di Matematica\\
Universit\`a  di Roma Tre\\
 L.go S.Murialdo 1, 00146 Roma, Italy. }
\email{{\tt orlandi@mat.uniroma3.it}}
\date{\today}
\begin{abstract}
We add a 
random bulk term, modelling the interaction with the impurities 
of the medium,
to a standard functional in the gradient theory of 
phase transitions consisting 
of a gradient term with a double well potential.   
We  show that in $d\le2$  there exists, for almost all the realizations of the random bulk term,   a   unique    random macroscopic  minimizer. This result  is  in sharp contrast to the case when the random bulk term is absent. 
{In the latter case there are two minimizers which are (in law) 
invariant under translations in space.}
 \end{abstract}

\maketitle

\section{Introduction}

 Models where  a stochastic
contribution is added to the energy of the system 
naturally arise in condensed  matter physics where the
presence of the impurities causes the microscopic 
structure to vary from point to point. 

We are interested in functionals  which - without random perturbation-  model the free energy of a material with
two (or several) phases on a so called mesoscopic scale,  i.e a scale which is much larger than the atomistic  scale 
so that the 
adequate description of the state of the material is by a {\em continuous}
scalar order parameter $m:\ D\subseteq\R^d\to \R$.    The  minimizers of these functionals are functions $m^*$ 
representing  the states  or  phases  of the materials. 

The natural question that we pose is the following:  What  happens  to these minimizers  when   an  external, even very weak,  random force  is added to the deterministic functional?
Are there still  the  same number of minimizers, i.e will  the material always have the same number of states (or phases)?
Is there some significant difference in the qualitative properties of the material  when the randomness is added?  These are standard questions in   a calculus of variations framework.  
However, standard  techniques  applicable for deterministic
calculus of variation problems might  not give a satisfactory answer when  randomness is involved.   
In the case under consideration in this paper one needs to deal with  a  family of nonlinear functionals which are not convex  and not bounded uniformly from below. So one needs to find,  depending on the functionals,   a way to answer these questions.  It turns out that methods  used in statistical mechanics,    suitably modified,   might  give an answer to these problems in certain cases.   In the last years there has been a  quite intensive flux in both directions to built a bridge between techniques and methods  used in analysis and  calculus of variations and  those used in statistical mechanics, see for a surveys  on these  issues  \cite{P}.  This paper is in this context.  

The analysis of the asymptotic behaviour of random functionals has received considerable attention  
within a homogenization framework, we mention for example the work by G. Dal Maso and L. Modica,
 \cite{DMM1},\cite{DMM2}.  The techniques there are based on $\Gamma$-convergence from the analysis side and
the sub-additive ergodic theorem from the probabilistic side. 

More recently, A. Braides and A. Piatnitski, \cite{BP}, 
studied a random optimization problem motivated by  problems in mechanics, which requires 
techniques from percolation theory.

Problems from solid mechanics lead naturally to the mathematical analysis of the asymptotic behaviour of random functionals, see e.g. \cite{Be1}, \cite{Be2}, \cite{BBN}.

The limit under consideration here, which requires techniques from statistical mechanics,  is different from the problems mentioned previously due to two averaging effects taking places simultaneously:
The singular  limit of a functional with several ground states (minimizers), and the averaging
over a random perturbation. 
The functional we will study  here 
consists of three competing parts:
An "interaction term" penalizing spatial changes in $m,$
a double-well potential $W(m),$
i.e. a nonconvex  function which has exactly two minimizers,
for simplicity $+1$ and $-1,$ modelling a two-phase material,
and a term which couples $m$ to a random field  $\theta g(\cdot,\omega)$ with mean zero,
variance  
$\theta^2$ 
and unit correlation length; i.e a term    
which   prefers at
each point in space one of the two minimizers of $W(\cdot)$ and
breaks the translational invariance, but is
"neutral" in the mean.  
A standard choice with the aforementioned
properties is
$$
\hat G ( m, \om ):=\int_D\left(|\nabla  m(y)|^2+
W(m(y)) - \theta g(y,\omega) m(y) 
\right )\rm {d }y.
$$  
We are, however, interested in a so-called {\em macroscopic} scale, which
is coarser than the mesoscopic scale. Therefore we 
rescale space with a small parameter
$\eps.$ If $\Lambda=\eps D$ and $u(x)=m(\eps^{-1}x),$ we obtain
$\hat G(m,\om)=\eps^{1-d}G_\e(u,\om),$ where 
\begin {equation} \label {funct1}
G_\eps  ( u,   \om, \La ):=\int_\Lambda\left(\e  |\nabla  u(x)|^2+
\frac 1 \e  W(m(x)) -   
\frac {\theta} \e    g_\e (x,\omega) m(x)
\right)\rm {d }x 
\end {equation} 
where $g_\eps$ has now correlation length $\eps.$
 We  are interested in determining the minimizers of this functional, the  asymptotic behavior  (as $\e \to 0$) of them  and their qualitative properties.  
 
  Due to the non-convexity of
the double-well potential,   the Euler-Lagrange
equation does not have an  unique solution.  

The $g$-dependent bulk term, can,    because of 
the scaling with $\eps^{-1}$, force a sequence $u_\eps$ to ``follow'' 
the oscillations of $g.$ This always happens in the form of 
bounded oscillations around the two wells of the double well potential.
In such a situation there are still two distinct minimizers.   But in
principle the $g$-dependent term could be strong enough to enforce
large oscillations,  so that   the minimizers
will  ``change well.''
In the periodic case it is possible to check on a deterministic
volume with a diameter of the order of the period whether the
minimizer ``changes well,'' i.e. creates a ``bubble'' of the other phase. 
The random case is quite different, because
there is no deterministic subset of $\Lambda$ such that the integral
of the random field over this subset equals zero for almost
all realizations of the random field - there are always {\em fluctuations}
around the zero mean.  
A set $A$ becomes the support of a bubble of the other phase
if the cost of switching to the other well, which can be estimated
by the Modica-Mortola result, see \cite{MM} and \cite{Mod}, as proportional to the boundary of $A,$ 
is smaller than the integral of the random field part over $A.$
As the correlation length is $\eps,$ a set
$A\subseteq\Lambda$ contains roughly 
$|A|\eps^{-d}$ independent random variables, where $|\cdot|$ denotes the
$d-$ dimensional Lebesgue measure of a set. By the central limit theorem,
fluctuations of order $\theta\sqrt{|A|}\eps^{d/2}$ are highly likely, but
the probability of larger fluctuations vanishes exponentially fast.
Therefore, using the isoperimetric inequality, the probability of $A$ being
the support of a bubble is exponentially small if 
\begin {equation} \label {A1}
c_d|A|^{(d-1)/d}\gg |A|^{1/2}\eps^{(d-2)/2}\theta,
\end {equation}
where $c_d$ is the isoperimetric constant.
In $d\ge 3$ this is asymptotically always the case 
for sets of diameter of order larger $\eps,$
or for sets of any size, provided  $\theta\to 0.$  
Dimension   $d=2$ and  $\theta$  small  is  a  critical case. 
In $d\ge 3$, although  \eqref {A1} holds for one {\em single}  bubble 
to  determine the  properties of the minimizers  one needs to ask if
{\em there exist} ``bubbles'' of the other phase.
These kind of problems were  discussed by the physics community 
in the  1980's for the  random-field Ising model.  
The question  was  to determine the dimension at which   the Random Field Ising model
would show spontaneous magnetization at low temperature and weak disorder.   This is closely related to the question
whether there are at least two distinct minimizers, one predominantly
$+$ and one predominantly $-$ for functional   \eqref {funct1}. 
  
This program  has been successfully   carried out   in a previous paper  by Dirr and Orlandi,   \cite {DO},
 in $d\ge 3$ and $ \theta\simeq \frac {1} { |\log  \e| } $.   
They  show  that,      $\Pr$- a.s with respect to the random field,  for any $\e>0$ there were  still two minimizers, which, unlike in the case $\theta=0,$ 
were not  constant functions $u(x) \equiv 1$ and $u(x)\equiv-1,$ but
 functions varying in $x$ and $\omega$ and the minimal
 energy  was  strictly negative. Further using $\Gamma-$convergence technique they  determined the cost of forming a {\em bubble} of one phase in the other one.
 These results were  obtained  under the strong assumption that  $\theta \simeq  \frac {1} { |\log  \e| }$.   We expect   by analogy with the Ising models with random field, 
 that for $\theta $ small but fixed, in $d\ge 3$ there are still two minimizers but they do not stay in one single well.  But so far there are no results in this case. 
 
Here we address the case when $d\le 2 $,  the strength of the random field $\theta$  is fixed.  We show that  when $d=1,2$  there exists  for almost all the realizations of the random field  an unique   {\it macroscopic}  minimizer  $ u^* (\cdot, \om)$ so that  denoting  $ Q(0)$ the unit cube  centered at the origin and $Q(z)=z+ Q(0)$, the unit cube centered in $z \in \Z^d$, 
$$   \E \left [ \int_{Q(z)} u^*(x, \cdot) dx \right ] = 0, \quad \forall z \in \Z^d, \quad d \le 2. $$ 
 Note that for $\theta=0$ and for sufficiently small {\em periodic} forcing 
there exist {\em two} minimizers (see e.g\cite{DLN1}), so the uniqueness of the minimizer is due to the
{\em random} nature of the perturbation.  

The {\em proof} of this is   based on  the following   steps.  We   prove  first that there exists two  {\it macroscopic extremal }  minimizers  $v^\pm (\cdot, \om)$   so that any other
macroscopic minimizer  satisfies $v^- (\cdot, \om) \le  u^* (\cdot, \om) \le v^+ (\cdot, \om)$. 
By a standard argument  then we  show that for any $\La \subset \R^d$  and for a positive constant $C$   \begin {equation} \label {G10}   \left [ G_1  ( v^+,   \om, \La )-  G_1  ( v^-,   \om, \La  ) \right ] \le C  |\La| ^{\frac {d-1} d}, \quad \forall \om \in \Omega.\end {equation}
Then we show that     
$$   F_n (\om):=    \E \left [ G_1  ( v^+,   \om, \La_n )-  G_1  ( v^-,   \om, \La_n ) | \BB_{\La_n} \right ]  $$  has significant fluctuations, with variance of the order of the volume.
Namely we show that
$$  \E  \left [   F_n (\cdot )\right ] =0, $$
and 
 \begin{equation} \label{mars3}
 \liminf_{n \to \infty} \E  \left [ e^{t  \frac {F_n} {\sqrt { \La_n}} } \right ] \ge e^{\frac {t^2 D^2}  2}.   \end{equation}   
%\begin {equation} \label {G11}    
%  \lim_{ n \to \infty} \frac 1 {\sqrt {| \Lambda_n|}} \left [  F_n (\cdot )  \right ]  \stackrel {D} %{=} Z,     \end {equation}
%where  $  \stackrel {D} {=}$ stands for convergence   in  distribution and   $ Z$ is  a %gaussian   random variable with mean 0 and variance 
% $b^2$. 
This holds in all dimensions.  
But in $d\le 2$ this generates a contradiction  with the  bound \eqref  {G10}, unless $D^2=0$. 
 When   $ D^2 =0$  we show that  
  $ M=  \E [ \int_{Q(0)}  v^+] -   \E [  \int_{Q(0)}  v^-] =0$.   
Further, we  show  that  $ \E [ \int_{Q(0)}  v^+] \ge \E [  \int_{Q(0)}  v^-] $, therefore  $ \E [ \int_{Q(0)}  v^+] =   \E [  \int_{Q(0)}  v^-] =0$. 
%The hard part is to  prove \eqref {mars3} and link   $D^2$ 
%to  $M$.  
 The probabilistic  argument  has been already applied   by
 Aizenman and  Wehr, \cite {AW}, in the context of Ising spin systems with random external field,
 see also the book by Bovier, \cite{B}, for a survey on  this  subject.

\section{Notations and Results}
 \subsection{The functional} 
The ``macroscopic'' space is given by  
$\Lambda:=[-\frac 12 ,\frac 12 ]^d,$  the
 $d-$ dimensional unit cube centered at the origin.  
The ratio between the macroscopic and the ``mesoscopic'' scale is given
by the small parameter $\e$. 
%For certain probabilistic estimates it will be necessary to
%require $\eps$ to be in a countable set, e.g. 
%$\eps=2^{-n},$ $n\in \N.$
 The disorder or random field is constructed with the help of
a family  $\{g(z,\omega)\}_{z\in \Z^d}$, $ \omega \in \Omega$   of 
independent,
identically distributed random variables which are absolutely continuous with respect to the
Lebesgue measure. The law of this family of random variables will be denoted by 
$\Pr$  and by $\E[\cdot]$ the mean with respect to $\Pr$.  
We assume that
\begin {equation} \label {eq:ass}  -1 \le  g(z, \om)   \le1, \quad  \forall \om \in \Omega, \quad \E[ g(z)]=0, \quad \E [g^2(z)]=1,   \quad  \forall z \in \Z^d.   \end {equation}
We denote by $ \|g\|_\infty = \sup_{z} |g(z, \om)|$. By   assumption   $\|g\|_\infty =1$, but to trace the dependence on  it   we   write  the explicitly dependence. 
The boundedness  assumption is  not essential. 
Different  choices 
of $g $ could be handled by minor modifications provided  $g$ 
is  still a random field with finite correlation length, 
invariant under (integer) translations and such that
$g(z,\cdot)$ has a symmetric distribution,  absolutely continuous w.r.t the Lebesgue measure and  $ \E[ g(z)^{2+ \eta}] < \infty$, $ z \in \Z^d$ for $\eta >0$. 
The  method  does not apply  when   $g$ has atoms.  
In  Ising spin systems, the uniqueness of the minimizer  
may fail if the distribution of $g$ has atoms, 
see \cite{AW}.  %with compact support.

We denote by  $ \BB$ the product $\sigma-$algebra and by $ \BB_\Lambda$, $  \Lambda \subset \Z^d$,  the $\sigma-$ algebra generated by $\{ g(z, \om): z \in \Lambda \}$.  
In the following we  often identify  the random field $ \{g(z, \cdot): z \in \Z^d\}$ with the coordinate maps  $ \{ g(z, \om)= \omega (z): z \in   \Z^d\}$. 
  To use ergodicity properties of the random field it is convenient 
to  equip  the  probability space $ (\Om, \BB, \Pr) $    with some extra structure.
First, we define the action $T$  of the translation group  ${\Z}^d$  on $ \Om$. We will assume that $ \Pr$  is invariant under this action and that the dynamical system  $ (\Om, \BB, \Pr, T)$ is stationary and ergodic.
In our model 
the action of  $T$  is  for $y \in \Z^d$  
 \begin {equation} \label {parisv1} (g (z_1,  [T_y \om]), . . . , g (z_n,  [T_y \om])) = ( g ( z_1+y,  \om) , . . . ,  g (z_n + y,  \om )). \end {equation}
The disorder or random field in the functional will be obtained
by a rescaling 
of $g$ such that the correlation length is order $\eps$ and the
amplitude grows as $\eps\to 0.$
To  this end define for $x \in \Lambda$ a function
$g_\e (\cdot, \omega)\in L^{\infty}(\Lambda)$  by
\begin{equation}\label{gfrombern}
g_\e (x, \omega):=\sum_{z\in \Z^d}g(z,\omega)
\1_{\e (z+[-\frac 12 ,\frac 12 ]^d)\cap \Lambda }(x),
\end{equation} where for any Borel-measurable set $A$
$$
  \1_A (x):= \begin{cases} &
1, {\rm if\ } x
\in A\\ &  0 \ {\rm if\ } x \not\in A . 
\end{cases}
$$
 The potential  $W$  is a so-called  ``double-well potential:''

\noindent
{\bf  Assumption (H1) }  $ W \in C^2(\R)$, $W\ge 0$, $ W(s)=0$ iff $s \in
\{-1,1\}$, $W(s) =
W(-s)$ and
$W(s)$ is strictly decreasing in   $[0,1]$. Moreover there exists
$\delta_0$ and $C_0>0$ so that
\begin{equation} \label{V.1}   W(s) = \frac 1 {2 C_0} (s-1)^2 
\qquad \forall s \in (1-\delta_0,  \infty).
\end{equation}
Note that $W$ is slightly different from the standard choice
$W(u)=(1-u^2)^2.$ Our choice simplifies some proofs because it makes
the Euler-Lagrange equation  linear provided solutions stay in one ``well.''
These assumptions could be relaxed.
% but in order to keep the exposition
%reasonably short, we prefer to use stronger assumptions.    
 For   $u\in H^{1}(\Lambda)$ and any open set $A\subseteq
\Lambda$  
consider the following random functional

\begin{equation} \label{functional}
G_\e (u,\omega, A):=\int_A\left(\e|\nabla  u(x)|^2+
\frac{1}{\e }W(u(x))\right){\rm d }x
-\frac{1}{\e }   \theta \int_A  g_\e(x,\omega) u(x)\rm {d }x 
\end{equation}
where $ \theta >0$.

 Set $\e= \frac 1n$, $ n \in \N$,
  hence for any $n \ge 1 $ the mesoscopic space 
is defined as
$\Lambda_n:=[-\frac n {2} ,\frac n{2} ]^d.$ 
Consider   $v \in H^1(\Lambda_n)$ and  
denote       in mesoscopic coordinates 
   \begin{equation} \label{functional2}
G_1 (v,\omega,\Lambda_n):=  \int_{\Lambda_n} \left( |\nabla  v(x)|^2+
 W(v(x))\right)\rm {d }x
-  \theta \int_{\Lambda_n}  g_1 (x,\omega) v(x)\rm {d }x .
\end{equation}
 The relation between   \eqref{functional}   and  \eqref {functional2} is
 \begin{equation} \label{parism7}G_n ( u,\omega,\La)=  n^{-(d-1)}  G_1 
(v,\omega,\Lambda_n), \end{equation} 
 where $  v(x) = u (\frac 1 n x)$ for $x \in \Lambda_n$.

%The functional  \eqref{functional2}
%can be extended to a lower semicontinuous functional
%$G_n :\ L^1(\Lambda_n)\to \R\cup\{+\infty\}$ by defining
%$G_n (\Lambda_n, v, \om)= +\infty$ for any $v\not\in H^{1}%%(\Lambda_n) $ and $ \om
%\in \Om$.
 For  $n>1 $ fixed  and   $ \om \in \Om$  it 
follows in the same way as in the
case without random perturbation that the functional  $G_1 ( 
\cdot ,\omega) $
 is coercive and  weakly lower semicontinuous 
in 
$ H^{1}(\Lambda_n),$ so there exists at least one minimizer, 
see   \cite {Eva}, which is  
a random function in $H^1(\Lambda_n), $ i.e. 
different realizations of $ \omega$
will give different minimizers.

\begin {defin}  { \bf   Translational covariant states} %Let $ v: \R^d \times \Omega \to \R$ be a %function.   
We say that  the function  $ v: \R^d \times \Omega \to \R$  is   translational covariant if
 \begin {equation} \label {rome1}  v(x+y, \omega) = v(x, [T_{-y} \om]) \quad   \forall y \in \Z^d, \quad x \in \R^d. \end {equation}
\end {defin}

\subsection{Minimizers }

Our   main result is  the following. 
 \begin{thm} \label{min}  Take    $ d\le 2 $,    $ \theta $    strictly positive 
  and  $ u^*_n (\cdot,\omega) \in  \argmin_{w  \in   H^1(\La_n) }G_1 (w,\omega,\Lambda_n)$.   Then, $\Pr$ a.s. there exists  an unique  $u^* (\cdot, \om) $  defined as  
  $$\lim_{n \to \infty}   u^*_n (x, \om)= u^* (\cdot, \om) $$
   so that  
   \begin{itemize}
 \item  $u^* (\cdot, \om) $ is   translation covariant, see \eqref {rome1},  \item Lipschitz continuous in $\R^d$,   
\item   $| u^* (\cdot, \om)|  \le 1+ C_0 \theta \|g_1\|_\infty $  where $C_0$ is the constant in \eqref {V.1}. 
 \item 
 \begin {equation}  \label {ca1}
 \lim n^{-d}G_1( u^*_n (\cdot, \om),\omega, \Lambda_n)=
 \lim n^{-d} \left ( \inf_{H^{1}(\Lambda_n)}G_1(\cdot,\omega,\Lambda_n)\right ) =e
\end {equation}
where  $ e $ is  a deterministic value given in \eqref {gi1}.
  \item 
  $$   \E \left [   \int_{z + [-\frac 12, \frac 12 ]^d}    u^*  (x, \om) dx \right ] =0,   \quad   \forall z \in \Z^d. $$
  \end {itemize}
  \end{thm}
\begin {rem}  When  $\theta=0$ in \eqref {functional2},  i.e the random field is absent,   the minimum value is zero and 
there are  two minimizers, the constant   functions  identical 
equal to $1$ or to $- 1$.
\end {rem} 
\begin {rem}  In the case analyzed in \cite {DO}, $ d \ge 3$, $\theta\simeq  \frac {\tilde \theta}  {\log n}$, $ \tilde \theta \in (0,1)$,
there exists two minimizers
$$u^\pm  (\cdot, \om) = \pm 1 + v^* (\cdot, \om),  \quad \E [ v^* (x, \cdot)] = 0, \quad \sup_{x} |v^*(x, \om)| \le C_0 \tilde \theta \|g\|_\infty.  $$
\end {rem}

    \section{Finite volume Minimizers}
 In this section we   state      properties for minimizers of the following
 problem       $$ \min_{w \in H^1 (\La_n)} G_1(w,\omega, \Lambda_n).  $$
 These properties hold in all   dimension $d$ and for  any $ \om \in \Omega$.   The  volume  $\La_n$ is kept fixed  in all the section. Thus to    short notation we denote $ \La:= \La_n$, 
  state the results for any $d$ and $\om$ plays the role of a parameter.        We first show that   
to determine the minimizers of the functional  $G_1,$
it is sufficient to consider functions in  $H^1(\Lambda)$  
which satisfy a uniform $L^\infty$-bound:
\begin{lem} \label{A}
Assume   (H1).  
 For all $\om \in \Omega$,  for all 
$v \in H^1(\Lambda)$ and all $t >1+ 
C_0\theta  \|g\|_\infty,$ 
\begin{equation}\label{eqabove}
G_1(t \wedge (v\vee (- t)),\om, \La ) -G_1 (v, \om, \La)\ge  
  \int_{\Lambda_t}   \left (   C_0^{-1} (t-1)  -
 \theta \|g\|_\infty
\right ) (|v(y)|-t),  
\end{equation}
where $C_0$ is the constant in \eqref{V.1} and 
$ \Lambda_t= \{ y \in \Lambda: |v(y)| >t\}.$
In particular $G_1(t \wedge v\vee (- t),\om, \La ) <G_1 (v, \om, \La)$
unless $\Lambda_t=\emptyset.$
\end{lem}
\begin{proof} 
\begin{equation*} 
  \begin{split}   & G_1 (v,\om,\La)-
G_1(t \wedge v \vee (- t),\om, \La )  \ge
    \int_{\Lambda_t}   \left (   W(v(y))- W(t) \right )
dy
 -  \theta \int_{\Lambda_t} dy   g_1 (y,\om) [
v(y)- \operatorname{sign} (v(y)) t], 
 \end{split}
\end{equation*}
%The integrand is strictly positive provided 
%$ t>1+ C_0\alpha (\e) \theta  $.  
and from (H1) and the $L^\infty$-bound on $g$ we derive (\ref{eqabove}).
\end{proof}
This $L^\infty$ bound on the global minimizer implies Lipschitz-regularity.
Namely  a    minimizer   of $G_1 ( \cdot, \om)$ in   $ H^1(\Lambda)$   
is  % for all $\omega\in \Omega$ 
a  weak  solution  of the Euler-Lagrange equation
\begin{equation}  \label{EL.1} \begin{split}    
&   \Delta v =  \frac 1 {2} [W'(v)+\theta   g_1]    \quad 
\text{in } \Lambda,  \qquad   \om \in \Omega  \\ &
  \frac {\partial v} {\partial n} =0 \qquad   \hbox {on} \qquad   \partial 
 \Lambda.  
 \end{split}
\end{equation}
%with homogeneous Neumann boundary conditions.
% Since   $| g_1 (\cdot, \omega ) | \le  1$ for all 
% $ \omega \in \Omega$, by  
% standard elliptic regularity    any minimizer will be in 
%$W^{2,p}(\Lambda)$ for  
 %all $p <\infty$.   By  a standard result, see e.g. \cite{GT} page 151,
% $W^{2,p}(\Lambda)$ is continuously embedded in  
 %$C^m (\bar \Lambda)$ for  $0\le 
 %m < 2- \frac d p$.      
We have the following regularity result. 
   \begin{prop}\label{Lip}  
Let    
\begin{equation}  \label{Lip1} L_0 = 
C(d)[ \sup_{ \{s: s = v(r),r \in \Lambda \}}|W'(s)| + 
 \theta \|g\|_\infty],  \end{equation}
 where $C(d)$ is a positive constant dimensional depending.  
The  solution $v$ of  the Euler-Lagrange equation 
\ref{EL.1} satisfies  
 $$|v(r, \om)-v(r', \om )|< 
  L_0  |r-r'|, \quad r, r' \in \Lambda,  \quad \forall \om \in \Omega.  $$
\end{prop}
\begin{proof}
By  Lemma \ref{A},   
a global minimizer  $v$  satisfies the bound  $|v(r, \om)| \le  1+   
C_0\theta \|g\|_{\infty}   $    
for  $r \in \Lambda$  and $\om \in \Omega$.
Since   $| g_1 (\cdot, \omega ) | \le  1$ for all $ \omega \in \Omega,$
any minimizer will be a bounded solution of Poisson's equation with
a bounded right hand side.

    By  the  regularity theory for the
Laplacian (see \cite{GT})  
the solution  $v$ is   Lipschitz in $ \Lambda $ with a    
Lipschitz 
constant bounded by  the quantity
$L_0 $ defined in \eqref {Lip1}.
%= C(d)[\sup_{ \{s: s = v(x), x \in \Lambda \}}|W'(s)| + 
%\theta \|g \|_\infty]$, where $C(d)$ is a positive constant dimensional %depending.  
\end{proof}

  \vskip 0.5cm 

 The following lemma proves that minimizers  of $ G_1 (\cdot,\omega,\La)$ corresponding to ordered boundary conditions  on $ \Lambda$  are ordered as well, i.e  they do not intersect.     In particular if there exists  more than one minimizer   corresponding to the same boundary condition  they do not intersect. 
  \begin{lem} %[zero temperature FKG]
  \label{FGK} Let $w_1$ and $w_2$ be
functions
in $H^1(\Lambda)$ such that (in the sense of traces)  
$w_1\le w_2$  on $\partial \Lambda,$ and
 $$ u\in \argmin_{w-w_1\in H^1_0(\La)}G_1 (w,\omega, \La)  
\quad \hbox {and} \quad     v\in \argmin_{w-w_2\in H^1_0(\La)} 
G_1  (w,\omega, \La). $$ 
Then $u=v$ or  $|u(x)-v(x)|>0$ for all $x\in {\rm int}(\Lambda).$ 
If $w_1<w_2$ in an open set in $\partial \Lambda,$ 
then $u<v$ everywhere in ${\rm int}(\Lambda).$
\end{lem}
{\em Proof:}  The argument works for   
general functionals of the type  
$$E(w):=\int_{\Lambda}\left(|\grad w(x)|^2+f(w,x)\right) dx $$
where $\partial_{ww} f(w,x)$ is continuous on $\R\times \overline \Lambda.$
(Here we treat $\omega$ as parameter, i.e. it holds for any realization of the random field.)  

Note that
for any $H^1$-functions $u$ and $v$ 
$$
E(u\vee v)+E(u\wedge v)=E(u)+E(v). 
$$
If $ u \in  \argmin_{w-w_1\in H^1_0(\La)}E(w)$  and  $ v \in  \argmin_{w-w_2\in H^1_0(\La)}E(w)$ we have 
   $u\vee v=v, u\wedge v=u$ on $\partial \Lambda,$ 
and  by the minimization properties of $u$ and $v$ we get 
$E(u\vee v)\ge E(v),\ E(u\wedge v)\ge E(u).$ 
This implies that actually $E(u\vee v)= E(v),\ E(u\wedge v)= E(u),$ so
$u\vee v\in \argmin_{w-w_2\in H^1_0}E(u),$ 
$u\wedge v\in \argmin_{w-w_1\in H^1_0}E(u).$   
Obviously   the function 
$ m:=  u-u\wedge v\ge 0$  in   $\La$ and  in particular $m=0$ on   $\partial \Lambda$.  We have that   $ m$ in our context solves   
\begin{equation}  \label{a11} \begin{split}    
&   \Delta m =  \frac 1 {2} [ f'(u) - f' (u\wedge v)] =V(x)m    \quad 
\text{in } \Lambda,  \\ &
 m =0 \qquad   \hbox {on} \qquad   \partial 
 \Lambda 
 \end{split}
\end{equation}
with potential
$$
V(x)=\frac{1}{2}\frac{f'(u) - f' (u\wedge v)}{u - u\wedge v}
$$which is continuous because $f$ is twice continuously
differentiable  in its first argument.

Suppose there exists $x_0\in  \Lambda $ with $m(x_0)=0.$
By Harnack's inequality (See \cite{GT}, Thm. 8.20) for nonnegative
solutions to elliptic linear
equations, $\sup_{B_R(x_0)} m  \le C\inf_{B_R(x_0)}m  $ for any ball such
that $B_{4R}\subset\Lambda.$ The constant $C>0$ depends on the radius  $R$ and the
coefficients in \eqref {a11}. Hence $0\le m\le\sup_{B_R(x_0)}m=0,$  so $m\equiv 0$ on 
such a ball. It immediately follows that $m\equiv 0$ on ${\rm int}(\Lambda).$
Therefore in the interior of $\Lambda$ 
either $u=u\wedge v$ (in which case $u\le v$)  or $u>u\wedge v,$ 
i.e. $v<u.$ As minimizers are uniformly Lipschitz continuous,
the latter case is only possible if $u=v$ on $\partial \Lambda.$

Consider the first case:  
$\widehat m:=v-u\ge 0$.  We get, reasoning as before,  
$\Delta \widehat m =\widehat V(x) \widehat m$ 
with a uniformly continuous potential $\widehat V$.
Then arguing as above   
$\widehat m=0$ everywhere  or  $\widehat m>0$ everywhere.\qed 
 \vskip0.5cm
% Summarizing the previous results we have.
  %\begin{lem} \label{fvol}  [finite-volume minimizers]
%  Let $w_i\in H^1(\La_n)$, $i=1,2$  $ u \in  \argmin_{w-w_1\in H^1_0(\La_n)} %G_1(\Lambda_{n},w,\omega)$  then $u$ is   Lipschitz continuous with  Lipschitz %constant 
 % $L_0$ given in \eqref {Lip1}, independent on $n$.  Further if $v$ is any other %minimizer, $ v \in  \argmin_{w-w_2\in H^1_0(\La_n)} G_1(\Lambda_{n},w,\omega)$ %then if $w_1=w_2$ 
%  either $v <u$ in $int (\La_n)$ or  $u <v$ in $int (\La_n)$.     If $w_1<w_2$ in an open %set in $\partial \Lambda,$ 
%then $u<v$ in ${\rm int}(\Lambda).$
% \end{lem} 

   \section{Infinite volume covariant states}
 \begin{thm} \label{infvol}  [infinite-volume states]
For almost all   $\om \in \Omega$,    there exist  two functions $v^+(x,\omega),$ $v^-(x,\omega)$, $ x \in \R^d$,  having the following properties.
 \begin{itemize}
 \item  $v^{\pm} (\cdot, \om) $ is   Lipschitz continuous in $\R^d$ 
  \item     \begin {equation}  \label {eq:bound}   | v^{\pm} (\cdot, \om)|  \le 1+ C_0 \theta \|g_1\|_\infty,  \end{equation}    where $C_0$ is the constant in \eqref {V.1}.
 \item  \begin {equation}  \label {c1} v^+ (x, \om) = -v^- (x, -\om) \quad x \in \R^d.  \end{equation}
 \item    $v^{\pm} (\cdot, \om) $ are  translation covariant 
 \item
\begin {equation} \label {M1a} 
\lim n^{-d}\int_{\Lambda_n}v^\pm (x, \omega) {\rm d} x =m^\pm,
\end {equation} where $m^\pm =  \E \left [  \int_{   [-\frac 12, \frac 12]^d}   v^\pm(x, \cdot )  {\rm d} x \right ] $, and $m^+=-m^- \ge 0$. 
  \item 
 \begin {equation}  \label {M5}
 \lim n^{-d}G_1(v^+,\omega,\La_n)=
 \lim n^{-d}G_1(v^-,\omega,\La_n)= \lim n^{-d} \inf_{H^{1}(\Lambda_n)}G_1(\cdot,\omega,\La_n)=e
 \end{equation}
 where $e$ is deterministic value given in \eqref{gi1}. 
 \item  Let $ \bar w_n (\cdot, \om) \in   \argmin_{  H^1(\La_n)} 
 G_1(v, \om,\La_n)$ then  
%$\Pr$ a.s. there exists  
%$\bar w (\cdot, \om)= \lim_{n} \bar w_n (\cdot, \om)$    so that 
\begin{equation}  \label {diseq1}  v^- (x, \omega) \le  
\bar \liminf_{n\to\infty} w_n (x , \om) \le\limsup_{n\to\infty} w_n (x , \om) 
\le v^+ (x, \omega), \quad x \in \R^d. \end{equation}
 \end{itemize}
 \end{thm}
\begin {proof}  We start proving the existence.
Consider the following   boundary problems. For $z \in \Z^d$, $ C= C_0 \|g\|_\infty $  where $C_0$ defined in \eqref {V.1}
 \begin{equation} \label{m1} \inf_{ (v-(1+C\theta)) 
 \in  H^{1}_0(\Lambda_n+z)}G_1(v, \om,z+\Lambda_n),
\end{equation}
\begin{equation} \label{m2} \inf_{ (v+1+C\theta) 
 \in  H^{1}_0(\Lambda_n+z)}G_1(v, \om,z+\La_n).
\end{equation}
Denote by  $ v_n^{z,+} := v_n^{z,+}(\cdot, \om) $  the maximal minimizer of \eqref {m1}  and by    $ v_n^{z,-}:= 
v_n^{z,-}(\cdot, \om) $  the
minimal minimizer of \eqref {m2}.   
If $z=0$ we write $v^\pm_n.$  For each $n>0$ and for each $\om \in \Omega$  there exists at least one minimizer of problems \eqref {m1} and
 \eqref {m2}  by lower semicontinuity and coerciveness. 
By  Lemma  \ref {A},  $v_m^+\le 1+C\theta$
on $\partial \Lambda_n$ for $m>n$.  Lemma \ref{FGK} implies
that for any $x$ and $\omega$   (and $n>n_0(x)$ ) the sequence
$\{v^+_n(x)\}_n$  is decreasing.      Moreover it is bounded from below
by $-1-C\theta.$  Hence, reasoning in a similar manner for $v^-_n,$
$$
v^\pm(x,\omega):=\lim_n v^\pm _n(x,\omega)
$$exists and is measurable as function of $\om$. As the $v^\pm_n$ are bounded and
minimizers, they are uniformly bounded and uniformly Lipschitz  on each fixed cube $A$ which does not depend on $n$, see  Proposition \ref {Lip}.
This implies that subsequences converge locally uniformly
to a Lipschitz function. As the entire sequence converges pointwise,
the limit of any subsequence must coincide with $v^\pm,$ which is therefore Lipschitz. 
The same argument for general $z$ yields
monotone limits $v^{z,\pm}.$ 

To show \eqref {c1} we note that
 \begin{equation} \label{c2} \inf_{ (v- (1+C\theta))
 \in  H^{1}(\Lambda_n)}G_1(v, \om, \Lambda_n) =  \inf_{ (v- (1+C\theta))
 \in  H^{1}(\Lambda_n)}G_1(-v, -\om,\Lambda_n)=  \inf_{ (w+ (1+C\theta))
 \in  H^{1}(\Lambda_n)}G_1(w, -\om,\Lambda_n).
\end{equation}
If $ \bar v(\cdot, \omega)\in \argmin_{ (v- (1+C\theta))
 \in  H^{1}_0}G_1(v, \om,\Lambda_n)$ the function 
 $$- \bar v(\cdot, \omega)= \bar w(\cdot, -\omega)\in \argmin_{ (w+ (1+C\theta))
 \in  H^{1}_0)}G_1(w, -\om,\Lambda_n)$$  so that
 $\bar v(\cdot, \omega)= - \bar w(\cdot, -\omega)$. 
 Therefore if $ \bar v(\cdot, \omega)$ is the  maximal minimizer of  $\inf_{ (v- (1+C\theta))
 \in  H^{1}_0}G_1(v, \om,\Lambda_n)$    $ \bar w (\cdot, -\omega)$ is 
the minimal minimizer of  $\inf_{ (w+ (1+C\theta))
 \in  H^{1}_0}G_1(w, -\om,\Lambda_n)$. 

To show the translation covariance, 
notice that, by \eqref{parisv1}
$$
v_n^{0,+}(0,\omega)=v^{z,+}_n(z, T_{-z}\omega).
$$ This implies the translation covariance if we can 
show that $v^{0,+}(0,\omega)=v^{z,+}(z, T_{-z}\omega).$
As the limit does not depend on the subsequence,
we know that $v^{+}=\lim v^{+}_{2^n}.$ As
for $n$ large $\Lambda_n+z\subseteq\Lambda_{2^n},$
we get $v^{z,+}_n(0)\le v^{0,+}_{2^n}(0)$ and $v^{z,+}(0)\le v^{0,+}(0).$
The opposite inequality follows in the same way. 

Next we want to show \eqref {M1a}.  We have 
\begin {equation}  \label {M2} \begin {split} 
 & \int_{\Lambda_n}v^\pm(x, \omega) {\rm d} x=  \sum_{z \in \Lambda_n \cap \Z^d} \int_{\{ z+ [-\frac 12, \frac 12]^d\}} v^\pm(x, \omega) {\rm d} x \cr &
 = \sum_{z \in \Lambda_n \cap \Z^d} \int_{   [-\frac 12, \frac 12]^d} v^\pm(T_z x, \omega) {\rm d} x = \sum_{z \in \Lambda_n \cap \Z^d} \int_{   [-\frac 12, \frac 12]^d} v^\pm(x, T_{-z}\omega) {\rm d} x. 
\end {split} \end {equation} 
Since $|v^\pm(x,  \omega)| \le C$,  by the  Birkhoff's ergodic theorem, see for example  \cite{GK},   we have $\Pr-$ a.s 
\begin {equation}  \label {M3aa} \begin {split} 
 & \lim\frac 1 {n^d}\int_{\Lambda_n}v^\pm(x, \omega) {\rm d} x =  \lim \frac 1 {n^d} \sum_{z \in \Lambda_n \cap \Z^d}  
\int_{   [-\frac 12, \frac 12]^d}     v^\pm(x, T_{-z}\omega)  {\rm d} x \cr &
= \E \left [  \int_{   [-\frac 12, \frac 12]^d}   v^\pm(x, \cdot )  {\rm d} x \right ] = m^{\pm}.
 \end {split} \end {equation} 
 Next we show \eqref {M5}.  By the covariance property of $ v^\pm (\cdot, \cdot)$ and the choice of the double well potential $W$  ($W$ does not depend on $x$) we have   
\begin {equation}  \label {M7}  
 G_1(v^+(\omega),\omega,\Lambda_n)=  \sum_{z \in \Lambda_n \cap \Z^d}  G_1(v^+ (\omega),\omega,z+ [-\frac 12, \frac 12]^d)=   \sum_{z \in \Lambda_n \cap \Z^d}  
G_1(v^+ (T_{-z}\omega) ,T_{-z}\omega,[-\frac 12, \frac 12]^d).
 \end {equation} 
Therefore, by  Birkhoff's ergodic theorem,  $\Pr-$ a.s 
\begin {equation}  \label {M8}  
 \lim\frac 1 {n^d}  G_1(v^+(\omega),\omega,\Lambda_{n})=  \E [ G_1(v^+(\cdot),\cdot,[-\frac 12, \frac 12]^d )]. 
 \end {equation} 
 Since  
 $$ G_1(v^+(\omega),\omega,\Lambda_{n})= G_1(-v^+(\omega),-\omega,\Lambda_{n}) =  G_1(     v^-(-\omega),-\omega,\Lambda_{n})$$
 we have 
 \begin {equation}  \label {gi1}   \E [ G_1(v^+(\cdot),\cdot,[-\frac 12, \frac 12]^d )] = \E [ G_1(v^-(\cdot),\cdot,[-\frac 12, \frac 12]^d )] = e. \end {equation} 
  To show the last  equality  of \eqref {M5}  note that
 if $ \bar w_n (\cdot, \om) \in   \argmin_{  H^1(\La_n)} 
 G_1(\cdot, \om,\Lambda_n)$ then
 \begin{equation}\label{est1}  
G_1( \bar w_n, \om,\Lambda_n) \le  G_1(v^+_n, \om,\Lambda_n).
\end{equation}
 Moreover, let the cut-off function 
$\psi: \R \to \R$ be nondecreasing, 1-Lipschitz and such that 
$\psi(x)=0$ for $x<0,$ $\Psi(x)=1$ for $x>2.$ Then 
$$\hat w_n:=\Psi\big(\dist(x,\R^d\setminus\Lambda_n)\big)\bar w_n+
\left(1-\Psi\big(\dist(x,\R^d\setminus\Lambda_n)\big)\right)v^+_n,
$$satisfies the boundary conditions of $v_n,$ hence
\begin{equation}\label{est2}
G_1(\hat w_n, \om,\Lambda_n) \ge G_1(v^+_n, \om,\Lambda_n).
\end{equation} Moreover an explicit calculation using the Lipschitz
bounds of the minimizers, $\Psi$ and the double well potential
together with the bounds on the random
field shows that
\begin{equation}\label{est3}
G_1(\hat w_n, \om,\Lambda_n)\le G_1( \bar w_n, \om,\Lambda_n)+Cn^{d-1},\quad \forall \om \in \Omega
\end{equation}
where $C>0$ depends only on the double well potential and
on the bound on the random field. (For details see proof of Lemma \ref{A1}.)

Taking (\ref{est1}-\ref{est3}) together, we obtain that
$\lim_n \frac 1 {n^d}G_1(\bar w_n, \om,\Lambda_n)=
\lim_n  \frac 1 {n^d} G_1(v_n^+, \om,\Lambda_n).$

%{\bf CE: We need to show \eqref {diseq1}}
It remains to show \eqref{diseq1}. Let $x,\ w_n$ be as in the statement, and
$n$ large enough so that $x\in\Lambda_n.$ 
Note that by Lemma \ref{A}, $v_n^-(y,\omega)\le w_n(y,\omega)
\le v_n^+(y,\omega)$ for all $y\in \partial\Lambda_n.$ So by Lemma \ref{FGK}
we get that $v_n^-(x,\omega)\le w_n(x,\omega)
\le v_n^+(x,\omega).$ \eqref{diseq1} follows by taking liminf and limsup.

\end {proof}
 
%\begin{rem}
%If 
%$$
%\int_{Q_1(0)}(v^+-v^-)=0\quad {\rm a.s.}
%$$ then
% $$
%\int_{Q_1(0)} u =0\quad {\rm a.s.}
%$$ for any infinite volume limit of minimizers on cubes with Dirichlet 
%or Neumann boundary conditions. Note that by the deterministic FKG
%inequality $v^+\ge v^-,$ so the vanishing of the integral implies
%that $v^+=v^-$ for almost all $x\in\R.$ Then, however, by the last
%statement of Theorem \ref{infvol},
%also $v^-(x)=u(x)=v^+(x)$ holds for almost all $x.$

%\end{rem}

Next we  bound  uniformly on   $ \om$ the   difference   between the energy of the maximal  $+$ minimizer and the minimal  $-$ minimizer. 
\begin{lem}\label {A1}  Let  $\om \in \Omega$, $u^+  \in \argmin_{v-(1+C_0 \theta) \in H_0 (\La ) } G_1(v, \om,\La )$ and $u^-  \in \argmin_{v+(1+C_0 \theta) \in H_0(\La) } G_1(v, \om,\La)$. 
 There exist a positive  constant $C$   depending on $\theta$ and  $C_0$, see \eqref {V.1} so that 
 \begin{equation} \label{m3}\left |  G_1(u^+ , \om,\La) - G_1(u^-, \om, \La)\right | \le C | \Lambda|^{\frac  {d-1} d}. 
\end{equation}    
 \end{lem}
 \begin {proof}  
Set 
  $$ \tilde u (x,\om)=  \left \{ \begin {split} & u^+ (x, \om) \quad \hbox { for } \quad  x \in  \La  \setminus \{x \in \La : d(x, \partial \La) \le 1\} \cr & 
   u (x), \quad   u(x) + (1+C_0\theta) \in H^1_0 (\La),   \quad  x \in \La  : d(x, \partial \La) \le 1,    \end {split} \right .$$  where $u $ is an   arbitrary Lipschitz function, so that $ | \nabla u (x) | \le  2(1+C_0\theta) $   chosen to match the boundary conditions, i.e 
  $\tilde u (x,\om) \in H^1_0 (\Lambda)$.
  We have
       \begin{equation}   \begin {split}  G_1(\tilde u , \om, \La) &=
 G_1(u^+, \om, \La) + \int_{  \{x \in \La : d(x, \partial \La) \le 1\} } \left [ 
  \left( |\nabla  \tilde u(x)|^2+
 W(\tilde u(x))\right)  -   \left( |\nabla    u^+(x) |^2+
 W(u^+(x))\right)  \right ]  \rm {d }x \cr &
+   \theta  \int_{  \{x \in \La : d(x, \partial \La) \le 1\} }   g_1 (x,\omega) \left [  \tilde u(x)- u^+(x)  \right ] \rm {d }x \cr & \le
G_1(u^+, \om) + | \Lambda|^{\frac  {d-1} d} 
  \left [C   +  4\theta  \|g_1 \|_\infty  \right ], 
\end {split}  \end{equation}    
where $C= C(C_0, \theta)$ is a positive constant which might change from an occurrence to the other.  
Obviously
$$ G_1(u^-, \om,  \La) \le  G_1(\tilde u , \om, \La). $$
Therefore
$$ G_1(u^-, \om, \La)-  G_1(u^+, \om,\La) \le    | \Lambda|^{\frac  {d-1} d}   \left [C  +  4\theta  \|g_1 \|_\infty  \right ]. $$
Similarly one can show that
$$ G_1(u^+, \om,\La)-  G_1(u^-, \om,\La) \le   | \Lambda|^{\frac  {d-1} d}  \left [C  +  4\theta  \|g_1 \|_\infty  \right ]. $$
Therefore  \eqref {m3}.  
\end {proof} 

\vskip0.5cm \noindent 
The  quantity next defined  plays a fundamental  role. 
 \begin{defin}  \label{def1}
Let  $v^\pm (\om)$ be  the infinite volume states constructed before.    Denote 
  \begin{equation} \label{m4}   F_n (\om):=    \E \left [  G_1(v^+(\om), \om, \La_n) - G_1(v^-(\om), \om, \La_n) | \BB_{\La_n} \right ]. 
 \end{equation}   
 \end {defin}
 \begin {rem}  By definition $F_n (\cdot)$ is $\BB_{\La_n}$ measurable and by the symmetry  assumption on the random field  $\{g(z,\cdot), z \in \Z^d\}$ 
 \begin{equation} \label{m9} \E\left [  F_n (\cdot)\right ] =0. \end{equation} 
 Namely   $ v^+ (x, \om) = -v^- (x, -\om) $ for   $x \in \R^d$.  This implies that 
  \begin{equation} \label{m8a}    G_1(v^+(\om), \om, \La_n)= G_1(v^-(-\om), -\om, \La_n) 
  \end{equation}
 and by  the symmetry of the random field  we get  \eqref {m9}. 
\end {rem}
Next we want to  quantify  how much $v^\pm (\om)$ changes  when  the random field is modified  only in one site, for example at the site $i$.   We introduce the following notation:  
$$  \om^{(i)}: \om^{(i)}  (z)= \om (z) \quad z \neq i, \qquad   \om= (\om(i),  \om^{(i)})   \quad i,z \in \Z^d. $$
The $v^+(\cdot, (\om(0), \om^{(0)}))$ is then  the state $v^+$  when  the 
random field  at the origin is   $\om(0)$,  and   $v^+(\cdot, (\om(0)-h, \om^{(0)}))$   the state $v^+$ when the 
random field at the origin is   $\om(0)-h$.  Same definition for the infinite volume  state $v^-(\cdot, (\cdot, \om^{(0)}))$ and  for  the finite volume minimizers      $v^\pm_n(\cdot,  (\cdot, \om^{(0)}))$.

Now we are able to state the
following lemma:  

\begin{lem}   \label {A2}   For    $ \La \subset  \R^d$, $ 0 \in \La$,      $h>0$ we have 
 \begin{equation} \label{LL.4}  \begin {split}
  \theta h\int_{Q_1(0)}v^+(\om(0), \om^{(0)}) {\rm d} x  & \ge
 G_1(v^+(\om (0)-h,\om^{(0)}),(\om(0)-h, \om^{(0)}), \La)- G_1(v^+(\om(0),\om^{(0)}),(\om (0), \om^{(0)}), \La) \cr & \ge
 \theta h \int_{Q_1(0)}v^+(\om (0)-h, \om^{(0)} )  {\rm d} x  \end {split}
 \end{equation}
where   $Q_1(0):=[-1/2,1/2]^d$.
The same  inequalities hold for   $v^-$. 
\end{lem}
 \begin {proof} Let $ \La_n$ be  a cube centered at the origin  so that $ \La \subset \La_n$.  Let  $v^+_n   $ be the maximal minimizer of 
 \begin{equation}  \label{v3a}  \inf_{ (v-(1+C\theta)) 
 \in  H^{1}_0(\Lambda_n)}G_1(v, \om,\La_n).
\end{equation}
Remark  that $v^+_n $   is measurable with respect to  the random field $g (z,\om)$,   $ z \in \La_n \cap \Z^d$. 
We have
\begin{equation} \begin{split}\label{v3b}  &
   G_1(v^+_n(\om (0),\om^{(0)}),(\om(0), \om^{(0)}), \La)
-G_1(v^+_n(\om (0)-h,\om^{(0)}),(\om(0)-h, \om^{(0)}), \La) \cr &=
  G_1(v^+_n(\om (0),\om^{(0)}),(\om(0), \om^{(0)}), \La)-   G_1(v^+_n(\om (0),\om^{(0)}),(\om (0)-h, \om^{(0)}), \La)\cr &+
  G_1(v^+_n(\om(0),\om^{(0)}),(\om(0)-h, \om^{(0)}), \La)- G_1(v^+_n(\om(0)-h,\om^{(0)}),(\om(0)-h, \om^{(0)}), \La).  
 \end {split}  \end{equation}
By explicit computation,  see \eqref{functional2}, we have  that
$$ G_1(v^+_n(\om (0),\om^{(0)}),(\om (0), \om^{(0)}), \La)-   G_1(v^+_n(\om (0),\om^{(0)}),(\om(0)-h, \om^{(0)}), \La)= -  h\theta  \int_{Q_1(0)}v^+_n (\om (0), \om^{(0)} ) dx. $$

  The  last line in \eqref{v3b}    is nonnegative, because $v^+_n(g(0)-h,\om^{(0)})$
is a minimizer of  \eqref {v3a} when  the random field  is $(g(0)-h, \om^{(0)} )$. Therefore 
% $$ G_1(v^+_n(\sigma,\om^{(0)}),(-\sigma, \om^{(0)}), \La)- G_1(v^+_n(-\sigma,\om^{(0)}),(-\sigma, \om^{(0)}), \La) \ge %0 .$$ 
%So we have 
 $$G_1(v^+_n(\om(0)-h,\om^{(0)}),(\om(0)-h, \om^{(0)}), \La) - G_1(v^+_n(\om (0),\om^{(0)}),(\om (0), \om^{(0)}), \La)\le
    h \theta  \int_{Q_1(0)}v^+_n (\om (0), \om^{(0)} ) dx .$$
By splitting   
   \begin{equation*} \begin{split}  &
   G_1(v^+_n(\om (0),\om^{(0)}),(\om (0), \om^{(0)}), \La)
-G_1(v^+_n(\om (0)-h,\om^{(0)}),(\om (0)-h, \om^{(0)}), \La) \cr &=
  G_1(v^+_n(\om (0),\om^{(0)}),(\om (0), \om^{(0)}), \La)-   G_1(v^+_n(\om (0)-h,\om^{(0)}),(\om (0), \om^{(0)}), \La)\cr &+
  G_1(v^+_n(\om (0)-h,\om^{(0)}),(\om(0), \om^{(0)}), \La)- G_1(v^+_n(\om (0)-h,\om^{(0)}),(\om (0)-h, \om^{(0)}), \La)
 \end {split}  \end{equation*} we obtain in a similar way
 $$G_1(v^+_n(\om (0)-h,\om^{(0)}),(\om (0)-h, \om^{(0)}), \La) - G_1(v^+_n(\om(0),\om^{(0)}),(\om (0), \om^{(0)}), \La)\ge
   h \theta  \int_{Q_1(0)}v^+_n (\om(0)-h, \om^{(0)} ) dx .$$
   To pass to the limit 
note  that the cube $Q_1(0)$ remains fixed. 
Denote by $M$
the smallest integer such that $\Lambda\subseteq B_M(0)$, where $ B_M(0)$ is a ball centered at the origin of radius $M$.   
Let $\xi(r)$ for $r\ge0$  be a smooth  
cut-off function s.t. $\xi(r)=1$ for $r<M,$ $\xi=0$ for $r>2M.$ 
 Note that for
$n>2{\rm diam}(\Lambda)$ the function $$\hat v^+_n:=v^+_n\xi(|x|^2)$$
satisfies 
$$
\Delta \hat v^+_n=f_n(x)
$$ with $\sup_n \|f_n\|_{\infty}<C,$ $C$  depending on ${\rm
diam}(\Lambda),$ 
$\theta,$ the
double well potential,  the cut-off function and  $\|g\|_\infty$ the bound on the  random field.  
The first derivatives of $\hat v^+_n$ are,  away from the boundary,  
H\"older continuous
with any exponent $\alpha<1$ (take $\alpha=1/2$ for definiteness) 
and H\"older norm bounded uniformly in $n$ with a bound
depending only on $C$. 
(See \cite{GT}, Thm. 3.9.  Note that this interior
estimate is applicable, because our domain is a ball containing
$2\Lambda,$ so the square $\Lambda$ is contained in the interior.)
So an application of Arzela Ascoli's Theorem gives that for a
subsequence $v^+_n$ and $\grad v^-_n$ converge uniformly. By
Lebesgues's Theorem on dominated convergence, 
we may pass to the limit under the integral and the claim is shown.

%Multiplying by $-1$ the inequality we obtain the upper bound.
%Replacing  $ \sigma $ with $-\sigma$ we get the lower bound. 

The  corresponding statement for
$v^-$ are proved in the same way.
 \end{proof}  
 \begin{rem} \label {R1}  From  Lemma \ref {A2} we have that
 $$ \om(0)\mapsto \int_{Q_1(0)}v^+(\om(0), \om^{(0)}) {\rm d} x % \ge \int_{Q_1(0)}v^+(-1, \om^{(0)}) {\rm d} x. 
 $$ is nondecreasing.
% This means that (in the continuous case)
% $$ \frac {\partial } {\partial {g(0)} } ( \int_{Q_1(0)}v^+(g(0), \om^{(0)}) {\rm d} x)  >0 $$

 \end{rem}

 %$$
 %\frac{\partial G_1(v, \om)}{\partial g(x)}=\frac{\partial G_1(v, \om)}%{\partial g(x)}{|(u^*(\om),\om)}+
% \underbrace{\frac{\partial G_1(v, \om)}{\partial v (x)}{|(u^*(\om),%\om)}}_{=0}\frac{\partial u(\om)(x)}{\partial g(x)}
%$$where the last term  is zero as $u^*$ is a minimizer.
%Since 
%$$ \frac{\partial G_1(v, \om)}{\partial g(x)}{|(u^*(\om),\om)} = \theta  %u^*(x,\om) $$
%for any minimizer $u^*(x,\om)$ we have that
 %\begin{equation} \label{m90}    \frac{\partial  F_\e (\om)}{\partial g(x)} %= \theta [  u^+(x,\om)- u^-(x,\om) ] 
 %\end{equation} 
\begin{cor}\label{A2b} Let  $\om(i)$ be the random field in the site $i$ which has  probability distribution absolutely continuous w.r.t the Lebesgue measure. 
We have that  $G_1(v^+(\om),\om, \La)$ is ${\mathbb P}$-a.e. differentiable w.r.t to $\om(i)$ and 
$$
\frac {\partial G_1(v^\pm(\om),\om, \La) } {\partial {\om (i)} } = -\theta \int_{Q_1(i)}v^\pm(x,\om) {\rm d} x.
$$
\end{cor} 
\begin {proof}It is sufficient to consider the case $i=0.$
By applying Lemma \ref{A2} for $\om(0)$ and $\tilde \om(0):=\om(0)+h$ we see that left and right derivatives
exist and are equal if
$
s\mapsto \int_{Q_1(0)}v^+(s, \om^{(0)}) {\rm d} x 
$ is continuous at $s=\om(0).$ By Remark \ref{R1} this happens for Lebesgue almost all $s,$ hence by
the assumptions on the random field ${\mathbb P} $-a.e.
\end {proof}
  \begin{thm}\label{A3} 
  We have that 
  \begin {equation} \label {MM1} \lim_{ n \to \infty} \frac 1 {\sqrt {| \Lambda_n|}} \left [  F_n (\cdot )  \right ]  \stackrel {D} {=}  Z,     \end {equation}
where $Z$ stands for a  Gaussian   random variable with mean $0$ and variance 
$b^2$  with 
 \begin {equation} \label {MM2}  4 \theta^2 (1+ C_0 \theta \|g\|_\infty)^2 \ge   b^2 \ge \E \left [ \left ( \E   \left [ F_n | \BB(0)\right ] \right )^2 \right ]   \end {equation}
where  $ \BB(0)$ is  the  sigma -algebra generated by $g(0, \om)$ and $C_0$ is given in \eqref {V.1}.
 \end{thm}
 \begin {proof}  We prove the  theorem  invoking the  general result  presented in the appendix.
 In order to do so, we need to establish the relevant conditions.
   We decompose $  F_n $ as a martingale difference sequence. We 
  order the points in $\La_n \cap \Z^d$  according to the lexicographic ordering.   In the following 
    $ i \le  j$     refers  to the lexicographic ordering. 
  Any other ordering will be fine but it is convenient to fix one.
 We  introduce the family of increasing $\sigma-$ algebra 
 $  \BB_{n,i}$,  $ i \in  \La_n \cap \Z^d $ where 
 $  \BB_{n,i}$ is the $\sigma-$ algebra  generated by the random  variables $ \{g (z), z\in \La_n \cap \Z^d, z \le i  \}  $.  We denote by 
 $$  \BB_{n,0}= (\emptyset, \Omega),  \quad     \BB_{n,i} \subset   \BB_{n,j}   \qquad  i \le j, \quad i \in     \La_n \cap \Z^d,  \quad j \in  \La_n \cap \Z^d.  $$ 
           We split 
   \begin{equation} \label{m91}     F_n  = \sum_{i \in \Z^d \cap \La_n } \left ( \E[ F_n| \BB_{n,i}] - \E[ F_n | \BB_{n,i-1}]\right ):=  \sum_{i \in \Z^d \cap \La_n }  Y_{n,i}.    \end {equation} 
        By construction     $ \E \left [  Y_{n,i}\right ]= 0$ for $i \in \Z^d \cap \La_n $, 
        $ \E \left [  Y_{n,i} | \BB_{n,k} \right ]= 0$, for all  $0 \le  k \le i-1$.
        % and
        % $ \E \left [  Y_{n,i}   Y_{n,j} \right ]= 0$ for $i \neq j$. 
        %and $N=  |\Z^2 \cap \La_n|= (n+1)^d$.    
  Denote  
 \begin{equation}  \label{v2}  
V_n: = \frac 1 {   |\La_n \cap  \Z^d|}   \sum_{i \in \La_n \cap  \Z^d} \E \left [  Y^2_{n,i}  | \BB_{n,i-1} \right ] .    \end {equation}  
In Lemma \ref {d2}  stated and proven below we show that  $ V_n \to b^2$  in probability  and $b^2$  satisfies \eqref{MM2}.   
In Lemma \ref {LL1}   stated and proven below we show that   for any $a>0$ 
 \begin{equation}  \label{v3}    U_{n} (a): = \frac 1 {   |\La_n \cap  \Z^d|}   \sum_{i \in   \La_n \cap  \Z^d }  \E [  Y^2_{n,i} 1_{\{  |Y_{n,i}| \ge  a  \sqrt { |\La_n \cap  \Z^d|}\}}  |  \BB_{n,i-1} ]   \end {equation} 
converges  to $0$ in probability.  
We can then invoke Theorem 5.1, stated in the appendix.   
   The correspondence to the   notation used in the appendix is the following. Identify $|\La_n \cap  \Z^d|$ with $n$,  
 $    \frac {F_n} {\sqrt { |\La_n \cap  \Z^d|} } \leftrightarrow  S_n $,  $\frac {Y_{n,i} }  {\sqrt { |\La_n \cap  \Z^d|} } \leftrightarrow   X_{n,i}$ and $\BB_{n,i}  \leftrightarrow  \FF_{n,i} $.  
 Then \eqref {MM1} is obtained. 
    \end   {proof}

\begin {lem}  \label {d2} Let  $V_n$  be the quantity  defined in \eqref {v2}. For all $\delta>0$
 \begin{equation}  \label{D1}  
\lim_{n \to \infty}  \Pr \left [   |V_n  -b^2| \ge \delta \right ] =0,
\end {equation} 
where   for   $W_0$ is defined in \eqref{g1}
\begin{equation}  \label{D2}   b^2=  \E \left [ W_0^2 \right ].
    \end {equation} 
Further 
\begin {equation} \label {MM2a} 4 \theta^2 (1+ C_0 \theta \|g\|_\infty)^2 \ge b^2 \ge \E \left [ \left ( \E   \left [ F_n | \BB(0)\right ] \right )^2 \right ],   \end {equation}
where $C_0$ is given in \eqref {V.1}.
\end {lem} 
\begin {proof}
The proof of \eqref {D1}  is  done by applying conveniently  the ergodic theorem.  We introduce new sigma-algebra $\BB_i^{\le}$ generated by the random fields $ \{ g(z, \om),  z \in \Z^d,  z   \le i \} $ where $\le $ refers to the lexicographic ordering. 
 Define for $ i \in \Lambda_n$
\begin{equation} \label{g1}  W_i[\omega]=  \E \left [G_1(v^+ (\om), \om,\La_n) - G_1(v^- (\om), \om, \La_n) | \BB_i^{\le}\right ] - \E \left [  G_1(v^+(\om), \om,\La_n) - G_1(v^-(\om), \om, \La_n) | \BB_{i-1}^{\le}\right ].  \end{equation}    
Note that      $W_i$ is a   random variable  depending on   random fields on sites   smaller  or equal than $i$ under   the lexicographic order. In particular it does not depend on the choice of the cube $\La_n$ provided $ i \in  \Lambda_n$.  To verify this statement notice that   $v^{\pm} (\om)$ does not depend on $\La_n$.  Further 
   denote  for   $ i \in \Lambda_n$,  $\omega= (\omega^<_i,  \omega(i), \omega^>_i)$ where $\omega^<_i= (\omega(j), j<i)$  and $\omega^>_i= (\omega(j), j>i)$,   
   $\tilde \omega= (\omega^<_i,  \tilde \omega(i), \omega^>_i)$ and  $  \omega(s)= (\omega^<_i,  s, \omega^>_i)$, $ s \in [-1,1]$. 
 We can write $W_i[\omega]$ as 
  \begin{equation} \begin {split} \label{g1b}  &W_i[\omega]=  \int \Pr (d\omega^>_i)  \Pr(d\tilde \omega (i)) \left [ 
  G_1(v^+ (\om), \om,\La_n) - G_1(v^+ (\tilde \om), \tilde \om, \La_n) \right ] \cr &
  -  \int \Pr (d\omega^>_i)  \Pr(d\tilde \omega (i)) \left [ 
  G_1(v^- (\om), \om,\La_n) - G_1(v^- (\tilde \om), \tilde \om, \La_n) \right ]. \end {split}
  \end{equation}    
 By Corollary \ref{A2b},   $  G_1(v^+ (\om), \om,\La_n)$ is   a.e.
differentiable w.r.t $\omega(i)$ with derivative depending only on the random field on $Q_1(i)$. 
Therefore one has 
\begin{equation} \begin {split}  
  G_1(v^+ (\om), \om,\La_n) - G_1(v^+ (\tilde \om), \tilde \om, \La_n)  & = \int_{\tilde \omega (i)}^{ \omega (i)} 
\frac {\partial } {\partial s}  G_1(v^+(\om (s)),\om (s), \La_n)  ds  \cr & =- \theta  \int_{Q_1(i)} dx
  \int_{\tilde \om(i)}^{\om(i)} v^+(x, \omega^<,s,\omega^>) ds.   \end {split}  \end{equation}    
  Similar considerations hold for the last  addend of \eqref {g1b}. 
  Hence     $W_i$ does not depend on the choice of $\Lambda_n,$ provided $Q_1(i)\subset \Lambda_n$,   \begin {footnote} {When the distribution of the random field is  not absolutely continuous with respect to the Lebesgue measure we  are able to prove  only  that    the  bounds of the discrete derivative $ \frac {\partial   G_1(v^+ (\om), \om,\La_n) } {\partial \omega (i)}$  are  independent of $\La_n$.  But   this obviously is not enough  to  show  that $W_i$ does not depend on the choice of $\Lambda_n$. } \end {footnote}.
 Note that from the translation  covariant properties of $v^\pm$ we have
    $$W_i [\om]= W_0[ T_{-i} \om]. $$
     By construction, see \eqref {m91},     for any $i$ provided $n$ large enough so that   $i \in \La_n$,  we have 
\begin{equation} \label{LL.2} Y_{n,i} = \E \left [  W_i | \BB_{\La_n} \right ].  \end{equation}
   Further by Corollary \ref {A2b} \begin{equation} \label{LL.5} |W_0 (\om)| \le 2 \theta (1+ C_0 \theta \|g\|_\infty),  \quad  \om \in \Omega  \end{equation}
   where $C_0$ is given in \eqref {V.1}. 
Applying   the ergodic  theorem we have that    in probability 
\begin{equation} \label{VV1}  \lim_{n \to \infty}  \frac 1 {|\La_n \cap  \Z^d|} \sum_{i \in  \La_n \cap  \Z^d }  \E \left [  W^2_i | \BB^{\le}_{i-1} \right ]=  \E \left [ W_0^2 \right ].  \end{equation} 
Set $ \E \left [ W_0^2 \right ]= b^2$. 
   Recalling the definition of  $ V_n$ given in \eqref {v2} the proof of   \eqref {D1}   is completed   if we show the following. 
  For any $ \delta >0$ 
     \begin{equation} \label{g2}
 \lim_{n \to \infty} \Pr \left [  | \E \left [ Y_{n,i}^2 | \BB_{n,i-1 } \right ]   -  \E \left [  W_i^2 | | \BB^{\le}_{i-1} \right ] | \ge \delta \right]  =0.  \end{equation} 
 We show \eqref{g2} applying Chebyshev's  inequality. 
We  split 
\begin{equation} \begin {split} \label{ LL.10}
&  \left \{  \E \left [ Y_{n,i}^2 | \BB_{n,i-1 } \right ]   -  \E \left [  W_i^2   | \BB^{\le}_{i-1} \right ]\right \} \cr &=
  \E \left [ Y_{n,i}^2 - W_i^2  | \BB_{n,i-1 } \right ] + \E \left [W_i^2  | \BB_{n,i-1 } \right ]  -  \E \left [ W_i^2 | | \BB^{\le}_{i-1} \right ].
 \end {split}  \end{equation} 
 Denote  $f_i=  \E \left [W_i^2  | \BB^{\le}_{i-1}  \right ] $,  $ R_n= R_n(i)= dist (i, \partial \Lambda_n)$ and
$\BB_{i+ [-R_n,R_n]^d}$ the $\sigma-$ algebra generated by the random fields in the box   $[-R_n,R_n]^d $ centered in $i$. 
 We have
\begin{equation}   \label{LL.11}
\E \left [ \left (  \E \left [W_i^2  | \BB_{n,i-1 } \right ]  -  \E \left [ W_i^2 | | \BB^{\le}_{i-1} \right ]  \right )^2  \right ]  \le \E \left [ \left (  f_0 - E [f_0| \BB_{[-R_n,R_n]^d} ]\right  )^2 \right ]:= b_1(R_n).  \end{equation} 
  When  $ \lim_{n \to \infty}   R_n =\infty $,    
we have  for any square integrable function 
$$ \lim_{n \to \infty}  b_1(R_n) = 0. $$
Further  by   \eqref {LL.2}    we have 
\begin{equation}   \label{ LL.12}
   \E \left [ \left |  \E \left [ Y_{n,i}^2 - W_i^2  | \BB_{n,i-1 } \right ]  \right | \right ]  \le  \E \left [ \left |  Y_{n,0}^2 - W_0^2    \right | \right ] \le\left ( E [W_0^2] \right )^\frac 12 \left ( E [ (  Y_{n,0} - W_0)^2] \right )^\frac 12.   \end{equation} 
  Arguing as before, see \eqref {LL.11},  we get  
  $$  \lim_{n \to \infty} E [ (  Y_{n,0} - W_0)^2]=0 $$
 proving \eqref {g2}. 
   To  get    \eqref {MM2a} we  
denote $ \BB(0)$ the  sigma -algebra generated by $g(0, \om)$ and   by Jensen's inequality we obtain
 $$  \E \left [ W_0^2 \right ]  = \E \left [ E [   W_0^2 | \BB(0)]  \right ]    \ge   \E \left [ \left (  \E \left [ W_0| \BB(0)\right ]\right)^2 \right ] . $$
By simple computation, taking in account that 
 $$ \E \left [  \E \left [  G_1(v^+, \om,\La_n) - G_1(v^-, \om, \La_n) | \BB_{-1}^{\le}\right ] | \BB(0)\right ] =0,$$ we have 
   \begin{equation}  \begin {split}   &\E \left [ W_0| \BB(0)\right ]      \cr &= \E \left [ \E \left [  G_1(v^+( \cdot), \cdot,\La_n) - G_1(v^- (\cdot), \om, \La_n) | \BB_0^{\le}\right ] - \E \left [  G_1(v^+, \om,\La_n) - G_1(v^-, \om, \La_n) | \BB_{-1}^{\le}\right ] | \BB(0)\right ]       \cr & =\E \left [ \E \left [  G_1(v^+ (\om), \om,\La_n) - G_1(v^- (\om), \om, \La_n) |  | \BB_{0}^{\le}\right ] | \BB(0)\right ] \cr &= 
 \E \left [ F_n |  \BB(0)\right ] . \end {split}   \end{equation}  
  The lower bound \eqref {MM2a} is proven.  
 \end {proof}
   \begin{lem} \label{LL1}  Let  $U_n(a)$ defined in \eqref {v3}. For any $a>0$ for any $ \delta >0$ 
   $$ \lim_{n \to \infty} \Pr \left [ U_n(a) \ge \delta \right] =0. $$
\end{lem}
\begin {proof}
By  Chebyshev's inequality   we have that
$$  \Pr \left [ U_n(a) \ge \delta \right]  \le \frac 1 \delta \E[U_n(a)]. $$ 
Next we show  $ \E[U_n(a)] \to 0$ for all $ a>0$. 
\begin {equation} \begin {split} &  \E [  U_n(a)] =  \frac 1 {   |\La_n \cap  \Z^d|}   \sum_{i=1}^{   |\La_n \cap  \Z^d|} \E \left [  Y^2_{n,i}1_{\{  |Y_{n,i}| \ge  a  \sqrt { |\La_n \cap  \Z^d|}\}}   \right ] \cr &  \le
 \frac 1 {   |\La_n \cap  \Z^d|}   \sum_{i=1}^{   |\La_n \cap  \Z^d|}\left ( \E [ Y_{n,i}^{2q}] \right )^\frac 1 {q} \left ( \Pr \left [ | Y_{n.i}| > a \sqrt { |\La_n \cap  \Z^d|}\right] \right )^{\frac 1 p}. 
 \end {split} \end {equation} 
By  Jensen inequality and definition  \eqref {LL.2} we have
$$   \E [ Y_{n,i}^{2q}]   \le \E [ W_0^{2q}],$$
which is a bounded quantity for all $q\ge 1$ since \eqref {LL.5}. 
Applying Cheybishev inequality  and arguing as before 
we have 
 $$ \Pr \left [ | Y_{n,i}| >   a    \sqrt { |\La_n \cap  \Z^d|} \right ] \le \frac {  \E [ W_0^{2}]} {a^2   |\La_n \cap  \Z^d|}, $$ 
which  for all $a>0$ tends to $0$ when $n \to \infty$. 
 \end {proof}

\begin{lem}   \label {A2c}   For    $ \La \subset  \R^d$, $ 0 \in \La$,     % %$\sigma\in [-1,1]$ uniformly distributed, then  
we have
$$
\frac {\partial} {\partial g(0)}  \E \left [ F_n| \BB(0)\right ] =   
  - \theta  \E \left [ \int_{ Q_1(0)} v^+ (x, \om) dx  | \BB(0) \right ] + \theta  \E \left [ \int_{ Q_1(0)} v^- (x, \om) dx  | \BB(0) \right ]
  $$
where   $Q_1 (0):=[-1/2,1/2]^d$.
Further 
$$ \E \left [  \frac {\partial} {\partial g(0)}  \E \left [ F_n| \BB(0)\right ]  \right ] =   - 2 \theta m^+. $$ 
\end {lem}
\begin {proof} We first give the idea by a formal computation.
\begin {equation} \label {L6} \begin {split}  & \frac {\partial} {\partial g(0)}  \E \left [ F_n| \BB(0)\right ]  = \frac {\partial} {\partial g(0)}  \E \left [ G_1( v^+, \om, \La) | \BB(0)\right ]-   \frac {\partial} {\partial g(0)}  \E \left [ G_1( v^-, \om, \La) | \BB(0)\right ]\cr &
= \E \left [ \frac{\partial G_1(v, \om,  \La)} {\partial g(0)}{|(v^+(\om),\om)}+
  \underbrace{\frac{\partial G_1(v, \om, \La)}{\partial v (0)}{|(v^+(\om), \om)}}_{=0}\frac{\partial v^+(0,\cdot)}{\partial g(0)} | \BB(0)\right ] \cr &
  - \E \left [  \frac{\partial G_1(v, \om, \La)} {\partial g(0)}{|(v^-(\om),\om)} +
  \underbrace{\frac{\partial G_1(v, \om, \La)}{\partial v (0)}{|(v^-(0), \om)}}_{=0}\frac{\partial v^- (0, \cdot )}{\partial g(0)} | \BB(0)\right ] \cr & =
  - \theta  \E \left [ \int_{ Q_1(0)} v^+ (x, \om) dx  | \BB(0) \right ] + \theta  \E \left [ \int_{ Q_1(0)} v^- (x, \om) dx  | \BB(0) \right ]
  \end {split} \end {equation}
  where the  terms are   zero as $v^\pm$  are  minimizers.
 The last equality is obtained since 
 $$ \frac{\partial G_1(v, \om)}{\partial g(0)}{|(v^\pm(\om),\om)} =- \theta   v^\pm(x,\om). $$
  Unfortunately, $v^\pm$ is not differentiable in the random field everywhere. Lipschitz-continuity in 
the field would be sufficient, but this is difficult to derive from the Euler-Lagrange equation because of the lack of convexity of the associated functional. A rigorous proof follows from \ref{A2b} after taking conditional expectations.
 Further, by  Theorem \ref {infvol},  we have
 \begin{equation} \label{m90} \begin {split} &  \E \left [  \frac {\partial} {\partial g(0)}  \E \left [ F_n| \BB(0)\right ]  \right ] =  -  \theta \E \left [  \E \left [ \int_{ Q_1(0)} v^+ (x, \om) dx  | \BB(0) \right ]  \right ]  \cr & +  \theta  \E \left [  \E \left [ \int_{ Q_1(0)} v^- (x, \om) dx  | \BB(0) \right ]\right ]  =  \theta [- m^+ + m^-]= -2 \theta m^+.
 \end {split}
 \end{equation} 
\end {proof} 
 % \begin{lem} \label{AIII}
%Let $\nu=\frac{1}{2}(\delta_1+(\delta_{-1})$ and let $g:[-1,1]\to \R$ be %a
%function. Then
%$$
%\int g^2(\eta)\nu(d\eta)\ge \frac{1}{4}
%\left(g(1)-g(-1)\right)^2
%$$\end{lem}
 %\begin{proof}  Denote
%$
%M:=g(1)-g(-1)
%$ Moreover, suppose w.l.o.g. $|g(-1)|\le |g(1)|.$ Then 
%$g(1)^2=(g(-1)+M)^2$ and by 
%Young's
%inequality
%$$
%\int g^2(\eta)\nu(d\eta)=g^2(-1)+2g(-1)\frac{1}{2}M+\frac{1}{2}M^2\ge
%\frac{1}{4}M^2.
%$$
% \end {proof}
 
  \begin {lem}  \label {d3} If 
\begin{equation} \label{d3a}  \E \left [ \left (  \E \left [ F_n| \BB(0)\right ]\right) ^2\right ] =0 \end {equation} 
 then    $m^+=m^-=0$, see for the definition    \eqref {m1}. 
\end {lem}
   \begin {proof} 
Denote  $f (\om (0)):=  \E \left [- F_n| \BB(0)\right ] $.   
Set $s= \om (0)$, \eqref {d3a}    can be written as 
 $ \int f^2(s) \Pr(ds) =0.$   By Lemma  \ref {A2c} and by bound  \eqref {eq:bound} in Theorem    \ref {infvol}  we have that $(1+C_0\theta)\theta  \ge f'(s) \ge 0$ almost
 everywhere.  
 This implies that 
 $f(s)=0$ for $\Pr$ almost all point of continuity of the distribution $g(0)$.
If $f(s)=0 $  for $\Pr$ almost all point of continuity of the distribution $g$, then $ f'(s)=0$  for  $\Pr$ almost all point of continuity of the distribution $g(0)$.   But if  $ f'(s)=0$  then  from Lemma \ref {A2c} we get $m^+=m^-=0$.   
\end {proof}

\vskip0.5cm 
\noindent  { \bf Proof of Theorem \ref {min} }  

%By  \eqref {diseq1}  of Theorem \ref{infvol}    we have 
%\begin{equation} \label{mars2}
%v^+ (x, \om) \ge u^* (x, \om)   \ge v^-(x, \om).\end{equation}  
%Therefore for all $z \in \Z^d$ 
%\begin {equation} 
%\label {order1}  \int_{ z+  [-\frac 12, \frac 12]^d}   
%v^- (x, \cdot )  {\rm d} x  \le  
%\int_{ z+  [-\frac 12, \frac 12]^d}   
%u^* (x, \cdot )  {\rm d} x   \le  
%\int_{ z+  [-\frac 12, \frac 12]^d}   
%v^+ (x, \cdot )  {\rm d} x. \end  {equation} 
%Further 
By   Lemma \ref {A1}  there exists $C=C(C_0, \theta)>0 $ so that 
\begin{equation} \label{m3a}\left |  G_1(v^+ , \om,\La) 
- G_1(v^-, \om, \La)\right | \le C | \Lambda|^{\frac  {d-1} d}. 
\end{equation}    
Applying  Theorem \ref {A3} we get the following lower bound on the Laplace transform of 
$F_n(\om)$ defined in Definition \ref {def1}:
 \begin{equation} \label{mars1}
 \liminf_{n \to \infty} \E  \left [ e^{t  \frac {F_n} {\sqrt { \La_n}} } \right ] \ge e^{\frac {t^2 D^2}  2}  \end{equation}   
 where we denote, see \eqref {MM2}  by
 $$D^2 =  \E \left [ \left ( \E   \left [ F_n | \BB(0)\right ] \right )^2 \right ] .$$
It is immediate to realize that \eqref {m3a}  and \eqref {mars1} contradict each other in $ d \le 2$ unless $D^2=0$.  On the other hand  when  $D^2=0$,  Lemma \ref {d3} implies 
\begin{equation}\label{integralsequal}
m^+= -m^-=    
\E \left [  \int_{   [-\frac 12, \frac 12]^d}   v^\pm(x, \cdot )  {\rm d} x \right ] =0. 
\end{equation} 
Now \eqref{diseq1} implies that $P$-a.s. 
$v^+(x,\omega)\ge v^-(x,\omega)$ for all $x\in R^2.$ This and 
\eqref{integralsequal} imply that $v^+(x,\omega)=v^-(x,\omega)$ a.s.

For each $n$ and for each $\om \in \Om$, by lower semicontinuity and 
coerciveness of the  functional $ G_1$,  there exists at least one minimizer   
$ w_n (\cdot,\omega)$ in  
$   \min_{w  \in   H^1(\La_n) }G_1 (w,\omega,\Lambda_n)$. 
By \eqref{diseq1} and the fact that $v^+=v^-,$ the functions  
$w_n$ converge pointwise
to a limit $u^*(x,\omega),$ and for the limit we have 
$u^*(x,\omega)=v^+(x,\omega)=v^-(x,\omega).$
The properties of the minimizer stated in \ref{min} therefore follow from
the corresponding properties of $v^\pm,$ see Theorem \ref{infvol}.

% Inequality \eqref {order1} implies 
%  $$ \E \left [  \int_{ z+  [-\frac 12, \frac 12]^d}   
%u^* (x, \cdot )  {\rm  d} x \right ] =0, \forall z \in \Z^d.  $$

    \qed

  \section {Appendix}

The  main tool to prove Lemma  \ref {A3} is the  following general result
which we reported from  \cite {HH}, see    Theorem 3.2 and Corollary 3.1  of \cite {HH}] .  The correspondence to the previous notation is
 $    \frac {F_n} {\sqrt { |\La_n \cap  \Z^d|} } \leftrightarrow  S_n $,  $\frac {Y_{n,i} }  {\sqrt { |\La_n \cap  \Z^d|} } \leftrightarrow   X_{n,i}$ and $\BB_{n,i}  \leftrightarrow  \FF_{n,i} $, 
see  \eqref  {m4}, \eqref {m91}. 
  \begin {thm}   Let $S_{n,i}$,  $i=1, \dots k_n$ be  a double array of   zero  mean martingales  with respect to the filtration  $\FF_{n,i}$, $ \FF_{n,i} \subset \FF_{n+1,i}$ $i=1, \dots k_n$ with $S_{n,k_n}= S_n$,   so that $S_{n,i}= \E[ S_n| \FF_{n,i}] $.   It is assumed that $k_n \uparrow \infty$  as $n \uparrow \infty$. 
Denote 
$$ X_{n,i}:= S_{n,i}- S_{n,i-1}, $$
$$ V_{n}= \sum_{i=1}^{k_n} \E [  X^2_{n,i}| \FF_{n,i-1}],  $$
 $$ U_{n,a}= \sum_{i=1}^{k_n} \E [  X^2_{n,i} \1_{\{ |[  X^2_{n,i}| >a\}} |  \FF_{n,i-1}]. $$
 Suppose that
 \begin {itemize}
 \item  for some constant $b^2$ and for  all  $\delta>0$,  $\lim_{n \to \infty} \Pr [ |V_n-b^2| \ge \delta] =0 $, 
 \item  For any $a>0$ for any $ \delta >0$ 
   $$ \lim_{n \to \infty} \Pr \left [ U_n(a) \ge \delta \right] =0, $$
    (Lindeberg condition)
 \end {itemize}
 then in distribution 
  $$ \lim_{n \to \infty}  S_n \stackrel {D} {=}  Z,$$
  where $Z$ is a random gaussian variable with mean equal to zero and variance equal to $b^2$. 
\end {thm}

\end{document}